\documentclass{amsart}
\usepackage{amsmath,amsthm,amsfonts,amssymb}
\usepackage{latexsym}
\usepackage{mathrsfs}
\usepackage[all]{xy}
\usepackage{url}
\usepackage[pdftex]{graphicx}
\usepackage{float}
\usepackage{tikz}
\usepackage{longtable}
\usetikzlibrary{decorations.markings} 
\usepackage{hyperref}
\usepackage{ upgreek }
\usepackage{amssymb}
\usepackage{amsmath}
\usepackage{tikz}
\usepackage{comment}
\usetikzlibrary{decorations.markings}

\makeatletter
\newtheorem*{rep@theorem}{\rep@title}
\newcommand{\newreptheorem}[2]{\newtheorem*{rep@#1}{\rep@title}
\newenvironment{rep#1}[1]{\def\rep@title{#2 \ref*{##1}}
\begin{rep@#1}}
{\end{rep@#1}}} 
\makeatother

\theoremstyle{plain}
  \newtheorem{thm}{Theorem}[section]
  \newreptheorem{thm}{Theorem}
  \newtheorem*{thm*}{Theorem}
  \newtheorem{prop}[thm]{Proposition}
  \newreptheorem{prop}{Proposition}
  \newtheorem{lem}[thm]{Lemma}
  \newtheorem{cor}[thm]{Corollary}
   \newreptheorem{cor}{Corollary}
  
\theoremstyle{definition}
  \newtheorem{definition}[thm]{Definition}

\theoremstyle{remark}
  \newtheorem{remark}[thm]{Remark}
\numberwithin{equation}{section}

\makeatletter
\@namedef{subjclassname@2020}{%
  \textup{2020} Mathematics Subject Classification}
\makeatother

\begin{document}

\title[Parabolic subgroups of two-dimensional Artin groups]{Parabolic subgroups of two-dimensional Artin groups and systolic-by-function complexes}

\author{Mart\'in Axel Blufstein}

\address{Departamento  de Matem\'atica - IMAS\\
 FCEyN, Universidad de Buenos Aires. Buenos Aires, Argentina.}
\email{mblufstein@dm.uba.ar}

\thanks{Researcher of CONICET}



\begin{abstract}
We extend previous results by Cumplido, Martin and Vaskou on parabolic subgroups of large-type Artin groups to a broader family of two-dimensional Artin groups. In particular, we prove that an arbitrary intersection of parabolic subgroups of a $(2,2)$-free two-dimensional Artin group is itself a parabolic subgroup. An Artin group is $(2,2)$-free if its defining graph does not have two consecutive edges labeled by $2$. As a consequence of this result, we solve the conjugacy stability problem for this family by applying an algorithm introduced by Cumplido. All of this is accomplished by considering systolic-by-function complexes, which generalize systolic complexes. Systolic-by-function complexes have a more flexible structure than systolic complexes since we allow the edges to have different lengths. At the same time, their geometry is rigid enough to satisfy an analogue of the Cartan-Hadamard theorem and other geometric properties similar to those of systolic complexes.
\end{abstract}

\maketitle


\section{Introduction}

Let $\Gamma_S$ be a finite simple graph with vertex set $S$ and edges labeled by integers $m\geq 2$. The label of the edge between $s$ and $t$ is denoted by $m_{st}$, and we put $m_{st}=\infty$ if there is no edge between them.  The \textit{Artin group} defined by $\Gamma_S$ is the group $A_S$ given by the following presentation:
$$\langle S \mid \underbrace{sts\cdots}_{m_{st} \text{ letters}}=\underbrace{tst\cdots}_{m_{st}\text{ letters}} \forall s,t \text{ such that } m_{st}\neq\infty\rangle.$$
 The graph $\Gamma_S$ is implicit in the notation $A_S$.

In this article we study the structure of the parabolic subgroups of two-dimensional Artin groups. By results of Charney and Davis \cite{CD1, CD2}, it is well-known that an Artin group $A_S$ is $2$-dimensional (i.e. it has geometric dimension $2$) if and only if for every triangle in the graph $\Gamma_S$ with edges labeled by $p$, $q$ and $r$ we have $\frac{1}{p} + \frac{1}{q} + \frac{1}{r} \leq 1$.

Given an Artin group $A_S$, by a result of Van der Lek \cite{V}, the subgroup generated by a subset $S'\subseteq S$ is isomorphic to $A_{S'}$, where $A_{S'}$ is the Artin group corresponding to $\Gamma_{S'}$, the full subgraph of $\Gamma_S$ spanned by $S'$. These subgroups are called the \textit{standard parabolic subgroups} of $A_S$, and their conjugates are the \textit{parabolic subgroups} of $A_S$. Thanks to the result of Van der Lek, parabolic subgroups have become a fundamental tool in the study of Artin groups. For example, the Deligne complex \cite{CD2} and an analogue to the curve complex of braid groups \cite{CGGW,MW} are simplicial complexes for Artin groups that are defined using parabolic subgroups.

One of the main questions regarding parabolic subgroups is whether they are stable under intersection. That is, is the intersection of an arbitrary family of parabolic subgroups a parabolic subgroup? The answer was already known to be affirmative for the intersection of standard parabolic subgroups \cite{V} and in the case of braid groups. A braid group on $n$ strands can be thought of as the mapping class group of a punctured disk $D_n$ with $n$ punctures. Its parabolic subgroups are in bijection with isotopy classes of nondegenerate, simple closed multicurves in $D_n$. The complement of each of these multicurves is a disjoint union of punctured disks in $D_n$. In \cite{FM} an intersection between these families of punctured disks is defined. This intersection corresponds, via the bijection, to the intersection between parabolic subgroups, and can be used to give an affirmative answer to the question. It was also known for graph products of groups, and in particular for right-angled Artin groups \cite{AM}. More recently, the answer was generalized to Artin groups of spherical type using Garside theory \cite{CGGW}. Combining this previous result with the structure of the Deligne complex, in \cite{MW} it was shown that the intersection of two parabolic subgroups of spherical type inside an FC-type Artin group is a parabolic subgroup of spherical type.

In \cite{CMV}  Cumplido, Martin and Vaskou used a geometric approach to solve this problem for Artin groups of large-type (i.e. those with $m_{st}\geq 3$ for all $s,t\in S$). They introduced a simplicial complex associated to an Artin group, called the Artin complex, on which the Artin group acts cocompactly and without inversions (see Section \ref{theartincomplex} for definitions) and proved the following:

\begin{thm}[\cite{CMV}, Theorem 11, Remark 15, Corollary 16]\label{parabolic}
Let $A_S$ be an Artin group and $X_S$ its Artin complex. If any time an element of $A_S$ fixes two vertices of $X_S$ it fixes pointwise a combinatorial path joining them, then
\begin{enumerate}
\item An arbitrary intersection of parabolic subgroups of $A_S$ is a parabolic subgroup of $A_S$.
\item If $P_1$ and $P_2$ are two parabolic subgroups of $A_S$ such that $P_1\subseteq P_2$, then $P_1$ is a parabolic subgroup of $P_2$.
\item For a subset $B\subseteq A_S$, there is a unique minimal parabolic subgroup of $A_S$ (with respect to the inclusion) containing $B$, called the parabolic closure of $B$.
\item The set of parabolic subgroups of $A_S$ is a lattice with respect to the inclusion.
\end{enumerate}
\end{thm}

Note that items (3) and (4) are direct consequences of item (2). These last facts were also proven in \cite{CGGW} in the context of spherical Artin groups. The proofs presented in \cite{CMV} are a geometrical reinterpretation of the latter. The second item had also been proven previously by Godelle in \cite{G}.

In order to show that large-type Artin groups satisfy the conditions of the theorem, Cumplido, Martin and Vaskou proved that Artin complexes are systolic in the sense of \cite{JS}, and used the fact that if a group $G$ acts without inversions on a systolic complex and fixes two vertices, then it fixes pointwise every combinatorial geodesic between them (\cite[Lemma 14]{CMV}). Systolic complexes were first introduced by Chepoi under the name of \textit{bridged complexes} in \cite{CH}. They were later discovered and investigated independently by Januszkiewicz and \'Swi\c{a}tkowski in \cite{JS}, and by Haglund in \cite{H}. Systolicity gives certain rigidity to the complex, which allows them to control its combinatorial geodesics.

In the same spirit as \cite{HO}, in this article we generalize  Cumplido, Martin and Vaskou's result to a broader class of two-dimensional Artin groups. This is accomplished by considering a geometric structure more flexible than systolicity. In Section \ref{lengthsystolic} we introduce systolic-by-function complexes, which generalize systolic complexes, and prove an analogue of the Cartan-Hadamard theorem. In systolic-by-function complexes we permit the edges to have different lengths. This flexibility allows us to consider a broader family of examples, while maintaining a rigid enough geometry. In Section~\ref{theartincomplex} we recall the construction of the Artin complex given in \cite{CMV}. Then, in Section \ref{parabolicsubgroups} we show that the Artin complexes of \textit{$(2,2)$-free} two-dimensional Artin groups, which are those whose defining graphs do not have two consecutive edges labeled by $2$, are systolic-by-function and prove the following result.

\begin{thm}\label{theorem}
Let $A_S$ be a $(2,2)$-free two-dimensional Artin group with $|S|\geq 3$ and $X_S$ its Artin complex. Then, if an element of $A_S$ fixes two vertices in $X_S$, it fixes pointwise a combinatorial path joining them. 
\end{thm}

As an immediate consequence, the results of Theorem \ref{parabolic} hold for all $(2,2)$-free two-dimensional Artin groups (the cases with less than 3 generators were established in \cite{CGGW,CMV}). In particular, we derive the main result of this article.

\begin{thm}\label{main}
Let $A_S$ be a $(2,2)$-free two-dimensional Artin group. Then the intersection of an arbitrary family of parabolic subgroups is a parabolic subgroup.
\end{thm}

Finally, at the end of Section \ref{parabolicsubgroups}, we study the \textit{conjugacy stability problem} for two-dimensional Artin groups that are $(2,2)$-free, by applying a very recent result of Cumplido \cite{C}.  A subgroup~ $H$ of a group $G$ is \textit{conjugacy stable} if, for every pair $h,h'\in H$ such that there exists $g\in G$ with $g^{-1}hg=h'$, there is $\tilde{h}\in H$ such that $\tilde{h}^{-1}h\tilde{h}=h'$. The conjugacy stability problem consists in deciding which of the parabolic subgroups of an Artin group are conjugacy stable. By applying the algorithm described in \cite{C}, as a consequence of Theorem \ref{main} we obtain the following result.

\begin{thm}\label{conj}
Let $A_S$ be a $(2,2)$-free two-dimensional Artin group and $A_{S'}$ a standard parabolic subgroup. Then $A_{S'}$ is not conjugacy stable if and only if there exist vertices $s,t$ in $\Gamma_{S'}$ that are connected by an odd-labeled path in $\Gamma_S$, but are not connected by an odd-labeled path in $\Gamma_{S'}$.
\end{thm}

Since conjugacy stability is preserved by conjugation, the previous theorem  solves the conjugacy stability problem for $(2,2)$-free two-dimensional Artin groups. Notice that Theorem \ref{conj} is a generalization of \cite[Theorem C]{CMV}.

\bigskip

\textbf{Acknowledgments:} I would like to thank Gabriel Minian for useful discussions and advice.


\section{Systolic-by-function complexes}\label{lengthsystolic}

In this section we define systolic-by-function complexes, which are a generalization of systolic complexes. We prove some basic properties and a local-to-global theorem analogous to the Cartan-Hadamard theorem. In Section \ref{parabolicsubgroups} we will make use of this geometric structure to prove Theorem~ \ref{theorem}.

\begin{definition}\label{length}
A \textit{length function} for a simplicial complex $X$ is a function $l:\text{edges}(X) \to [0,\frac{1}{2}]$ that assigns a real number between $0$ and $\frac{1}{2}$ to each edge of $X$, satisfying the two following conditions:
\begin{itemize}
\item the sum of the lengths of the three edges of any triangle is less than or equal to 1;
\item the triangle inequality holds. That is, given three edges $e_0,e_1,e_2$ that form a triangle, $l(e_i)\leq l(e_{i+1})+l(e_{i+2})$ (indices modulo 3).
\end{itemize}
A simplicial complex together with a length function is called a \textit{length complex}. 
\end{definition}

A \textit{cycle} in a (length) complex $X$ is a subcomplex $\sigma$ homeomorphic to $S^1$. We denote by $|\sigma|$ the number of edges in $\sigma$. The length of $\sigma$ is the sum of the lengths of its edges, and we denote it by~ $l(\sigma)$. A \textit{path} in $X$ is a subcomplex $\gamma$ homeomorphic to $[0,1]$. We define $|\gamma|$ and $l(\gamma)$ analogously.

A subcomplex $K$ of a simplicial complex $X$ is \textit{full} if any simplex of $X$ spanned by a set of vertices in $K$ is a simplex of $K$. A \textit{diagonal} in a cycle $\sigma$ in a simplicial complex $X$ is an edge of $X$ connecting two nonconsecutive vertices of $\sigma$. Thus, a cycle is full if and only if it has no diagonals and does not span a simplex. A simplicial complex $X$ is \textit{flag} if every set of vertices pairwise connected by edges spans a simplex of $X$.

We recall that given a simplex $s$ in a simplicial complex $X$, its \textit{link} $Lk_X(s)$ is the subcomplex of $X$ consisting of the simplices that are disjoint from $s$ and such that, together with $s$, span a simplex of $X$.

\begin{definition}\label{large}
A length complex $X$ is \textit{large} if it is flag and if every full cycle has length greater than or equal to 2. It is \textit{locally large} if the link of every vertex is large.
\end{definition}

It is clear from the definitions that a large length complex is locally large. This is because, since the complex is flag, the links of its vertices are flag and full cycles in the links are full cycles in the complex. The rest of this section is devoted to showing that the converse holds when $X$ is simply connected. This is a local-to-global theorem analogous to the classical result for systolic complexes~ \cite{JS}.

\begin{definition}\label{lsystolic}
A length complex $X$ is \textit{systolic-by-function} if it is connected, simply connected and locally large.
\end{definition}

\begin{remark}
A simplicial complex is systolic if and only if it is systolic-by-function with constant length function $l\equiv \frac{1}{3}$. In general, a simplicial complex is $k$-systolic if and only if it is systolic-by-function with constant length function $l\equiv \frac{2}{k}$.
\end{remark}

\begin{thm}\label{localtoglobal}
Let $X$ be a systolic-by-function length complex. Then $X$ is large.
\end{thm}

In order to prove this theorem, we will have to study the structure of diagrams over a systolic-by-function complex. A \textit{diagram} $\Delta$ in $X$ is a simplicial map $\varphi:M \to X$. If $M$ is a simplicial structure of a $2$-dimensional disk, we say that $\Delta$ is a \textit{disk diagram}. A simplicial map is called \textit{nondegenerate} if it is injective in every simplex.

\begin{lem}[\cite{JS}, Lemma 1.6]\label{existdiagram} Let $X$ be a simplicial complex, and $\sigma$ a homotopically trivial cycle in $X$. Then there exists a nondegenerate disk diagram $\varphi: D \to X$, which maps the boundary of $D$ isomorphically onto $\sigma$.
\end{lem}

Such a diagram is called a \textit{filling diagram} for $\sigma$. In a simply connected length complex, the previous lemma implies that every cycle has a filling diagram. To understand these diagrams, we will recall some basic notions of combinatorial curvature.

Let $X$ be a $2$-dimensional length complex. If $v$ is a vertex of $X$, its curvature is defined as 
$$\kappa(v) = 2-\chi(Lk_X(v)) - \sum_{e\in Lk_X(v)} l(e).$$
Here $\chi(Lk_X(v))$ denotes the Euler characteristic of the link of $v$. We define the curvature of
a $2$-simplex $f$ of $X$ as
$$\kappa(f) = \left(\sum_{e\in\partial f}l(e)\right) - 1,$$
where $\partial f$ is the boundary of $f$ and the sum is over its three edges. Note that the curvature of a face is always nonpositive. The following well known result can be found in \cite{BB,Wi}.

\begin{thm}[Combinatorial Gauss-Bonnet Theorem]\label{gaussbonnet}
Let $X$ be a $2$-dimensional length complex. Then
$$\sum_{f\in \text{faces}(X)} \kappa(f) + \sum_{v\in \text{vertices}(X)} \kappa(v) = 2\chi(X).$$
\end{thm}

Notice that the formula is not exactly equal to the one appearing in \cite{BB,Wi}. The size of an angle is the length of the side opposite to it, multiplied by $\pi$. We omit the factor of $\pi$ for simplicity.

\begin{definition}\label{minimal}
Let $\sigma$ be a cycle in a simplicial complex $X$. A filling diagram $\varphi:D\to X$ for $\sigma$ is \textit{minimal} if $D$ has the least amount of $2$-simplices among all filling diagrams for $\sigma$. 
\end{definition}

Observe that if $\varphi:D\to X$ is a minimal filling diagram for a cycle $\sigma$, it is nondegenerate: if an edge $e$ were mapped to a vertex, we could take the two triangles containing $e$, delete the interior of their union and glue the remaining four edges, thus obtaining a filling diagram for $\sigma$ with less $2$-simplices. Given a nondegenerate diagram $\varphi:M\to X$, where $X$ is a length complex, we can pullback the length function to $M$, so that $M$ is a length complex itself. The next two lemmas follow the ideas from \cite[Lemmas 2.5 and 2.6]{BM}.

\begin{lem}\label{nointerior}
Let $X$ be a large length complex and $\sigma$ a cycle in $X$ of length less than 2. Then there exists a filling diagram $\varphi:D\to X$ for $\sigma$ such that $D$ has no interior vertices. 
\end{lem}
\begin{proof}
We proceed by induction  on $|\sigma|$. If $|\sigma| = 3$, then the result follows by flagness. Now suppose $|\sigma|>3$. Since $X$ is large, $\sigma$ cannot be full. Then $\sigma$ has a diagonal $e$ that connects two nonconsecutive vertices of $\sigma$. This edge subdivides $\sigma$ in two paths, both with less than $|\sigma|-1$ edges. We call them $\sigma_1$ and $\sigma_2$. Attaching $e$ to both $\sigma_1$ and $\sigma_2$, we get two cycles with less number of edges than $\sigma$. By the triangle inequality $l(\sigma_i\cup e)\leq l(\sigma)$ for $i=1,2$. By inductive hypothesis, there exist filling diagrams without interior vertices for both cycles. Gluing these two diagrams along the two edges mapped to $e$ we obtain the desired diagram for $\sigma$. 
\end{proof}

\begin{lem}\label{nicediagram}
Let $\sigma$ be a homotopically trivial cycle in a locally large length complex $X$. Then for any minimal filling diagram $\varphi:D\to X$ for $\sigma$, $D$ is locally large when considered with the pullback length.
\end{lem}
\begin{proof}
Let $\varphi:D\to X$ be a minimal filling diagram for $\sigma$. Suppose there is an interior vertex $v$ of~ $D$ such that $Lk_D(v)$ is not large. Since $D$ is simplicial, $Lk_D(v)$ is a full cycle in $D$ (that has length less than $2$). Consider $\varphi(Lk_D(v))$ as a cycle in $Lk_X(\varphi(v))$. Since $X$ is locally large, $Lk_X(\varphi(v))$ is a large length complex. Thus, by Lemma \ref{nointerior}, there is a filling diagram $\psi: D'\to Lk_X(\varphi(v))\subset X$ for $\varphi(Lk_D(v))$ with no interior vertices. There are $|Lk_D(v)|$ closed $2$-simplices in $D$ that contain $v$. Since $\psi$ is a filling diagram for $\varphi(Lk_D(v))$ and $D'$ has no interior vertices, the number of $2$-simplices in $D'$ is $|\varphi(Lk_D(v))|-2<|Lk_D(v)|$. Therefore, if we replace the set of closed $2$-simplices of $D$ that contain $v$ by this new diagram, we obtain a filling diagram for $\sigma$ with less $2$-simplices, which is a contradiction. Hence $D$ is locally large.
\end{proof}

\begin{remark}
If $\varphi:M\to X$ is a nondegenerate disk diagram, then being locally large is equivalent to $\kappa(v)\leq 0$ for every interior vertex $v$ of $M$. This is because, for an interior vertex $v$, $\kappa(v) = 2-\sum_{e\in Lk_M(v)} l(e)$.
\end{remark}

Let $\varphi : D\to X$ be a disk diagram. The \textit{boundary layer} $L$ of $D$ consists of every vertex in the boundary of $D$, every edge incident to a vertex in the boundary, and every open $2$-simplex whose closure has a vertex in the boundary. Here we consider open $2$-simplices because we do not want edges that are not incident to a vertex in the boundary to be part of the boundary layer. Note that $L$ is usually not a simplicial complex. If $D$ has at least two interior vertices, and no edge connecting nonconsecutive vertices of the boundary, we define the following complex. Consider the simplicial complex $A$ constructed by taking the disjoint union of the vertices, edges and $2$-simplices (now closed) of the boundary layer of $D$ and identifying the boundaries of the closed $2$-simplices but only in the vertices and edges of the boundary layer of $D$ (see Figure \ref{figone}). Since $D$ has more than one interior vertex and no edge connecting nonconsecutive vertices of the boundary, $A$ is an annulus without interior vertices. We call $A$ the \textit{boundary complex} of $D$. It has two boundary components $\partial_1A$ and $\partial_2A$, the first of which is isomorphic to $\partial D$. If $D$ is a length complex, then $A$ is a length complex with the induced length. Note that if $D$ has exactly two interior vertices, its boundary complex $A$ is not a simplicial complex, since it has a double edge. However, it is easy to see that all the definitions and results can be adapted to this case.

\begin{remark}
If the disk had only one interior vertex, its boundary complex would be the disk itself. This is why that case is excluded from the previous definition.
\end{remark}

\begin{figure}[h]
\begin{tikzpicture}[scale=0.6]

\tikzstyle{point}=[circle,thick,draw=black,fill=black,inner sep=0pt,minimum width=4pt,minimum height=4pt]
    \node (1)[point] at (0,0) {};
    \node (2)[point] at (1,2) {};
    \node (3)[point] at (1.5,0.5) {};
    \node (4)[point] at (1,-2) {};
    \node (5)[point] at (3,2) {};
    \node (6)[point] at (4,0) {};
    \node (7)[point] at (3.5,-2.2) {};
    \node (8)[point] at (4,3.5) {};
    \node (9)[point] at (5,2.5) {};
    \node (10)[point] at (6,1) {};
    \node (11)[point] at (6.5,-2) {};
    \node (12)[point] at (6,3.5) {};
    \node (13)[point] at (6.5,2.5) {};
    \node (14)[point] at (7.5,0) {};
    \node (15)[point] at (8.3,2) {};
    
    \draw [fill=black, fill opacity=0.2] (0,0)--(1,2)--(3,2)--(4,3.5)--(6,3.5)--(8.3,2)--(7.5,0)--(6.5,-2)--(3.5,-2.2)--(1,-2)--cycle;
    
    \draw[ultra thick] (1)--(2);
    \draw[ultra thick] (1)--(3);
    \draw[ultra thick] (1)--(4);
    \draw[ultra thick] (2)--(3);
    \draw[ultra thick] (2)--(5);
    \draw[ultra thick] (3)--(4);
    \draw[ultra thick] (3)--(5);
    \draw[ultra thick] (3)--(6);
    \draw[ultra thick] (3)--(7);
    \draw[ultra thick] (4)--(7);
    \draw[ultra thick] (5)--(6);
    \draw[ultra thick] (5)--(8);
    \draw[ultra thick] (5)--(9);
    \draw[ultra thick] (6)--(7);
    \draw[ultra thick] (6)--(9);
    \draw[ultra thick] (6)--(10);
    \draw[ultra thick] (7)--(10);
    \draw[ultra thick] (7)--(11);
    \draw[ultra thick] (8)--(9);
    \draw[ultra thick] (8)--(12);
    \draw[ultra thick] (9)--(10);
    \draw[ultra thick] (9)--(12);
    \draw[ultra thick] (9)--(13);
    \draw[ultra thick] (10)--(11);
    \draw[ultra thick] (10)--(13);
    \draw[ultra thick] (10)--(14);
    \draw[ultra thick] (11)--(14);
    \draw[ultra thick] (12)--(13);
    \draw[ultra thick] (12)--(15);
    \draw[ultra thick] (13)--(14);
    \draw[ultra thick] (13)--(15);
    \draw[ultra thick] (14)--(15);
    
    \node[] at (2,3) {$D$};

    \begin{scope}[xshift=12cm]
    \tikzstyle{point}=[circle,thick,draw=black,fill=black,inner sep=0pt,minimum width=4pt,minimum height=4pt]
    \node (1)[point] at (0,0) {};
    \node (2)[point] at (1,2) {};
    \node (3)[point] at (1.5,0.5) {};
    \node (4)[point] at (1,-2) {};
    \node (5)[point] at (3,2) {};
    \node (61)[point] at (4,0.5) {};
    \node (62)[point] at (4,-0.5) {};
    \node (7)[point] at (3.5,-2.2) {};
    \node (8)[point] at (4,3.5) {};
    \node (9)[point] at (5,2.5) {};
    \node (10)[point] at (6,1) {};
    \node (11)[point] at (6.5,-2) {};
    \node (12)[point] at (6,3.5) {};
    \node (13)[point] at (6.5,2.5) {};
    \node (14)[point] at (7.5,0) {};
    \node (15)[point] at (8.3,2) {};
    
    \draw [fill=black, fill opacity=0.2] (0,0)--(1,2)--(3,2)--(4,3.5)--(6,3.5)--(8.3,2)--(7.5,0)--(6.5,-2)--(3.5,-2.2)--(1,-2)--(1.5,0.5)--(4,-0.5)--(6,1)--(6.5,2.5)--(5,2.5)--(4,0.5)--(1.5,0.5)--cycle;
    \draw [fill=black, fill opacity=0.2] (0,0)--(1.5,0.5)--(1,-2)--cycle;
    
    \draw[ultra thick] (1)--(2);
    \draw[ultra thick] (1)--(3);
    \draw[ultra thick] (1)--(4);
    \draw[ultra thick] (2)--(3);
    \draw[ultra thick] (2)--(5);
    \draw[ultra thick] (3)--(4);
    \draw[ultra thick] (3)--(5);
    \draw[ultra thick] (3)--(61);
    \draw[ultra thick] (3)--(62);
    \draw[ultra thick] (3)--(7);
    \draw[ultra thick] (4)--(7);
    \draw[ultra thick] (5)--(61);
    \draw[ultra thick] (5)--(8);
    \draw[ultra thick] (5)--(9);
    \draw[ultra thick] (62)--(7);
    \draw[ultra thick] (61)--(9);
    \draw[ultra thick] (62)--(10);
    \draw[ultra thick] (7)--(10);
    \draw[ultra thick] (7)--(11);
    \draw[ultra thick] (8)--(9);
    \draw[ultra thick] (8)--(12);
    \draw[ultra thick] (9)--(12);
    \draw[ultra thick] (9)--(13);
    \draw[ultra thick] (10)--(11);
    \draw[ultra thick] (10)--(13);
    \draw[ultra thick] (10)--(14);
    \draw[ultra thick] (11)--(14);
    \draw[ultra thick] (12)--(13);
    \draw[ultra thick] (12)--(15);
    \draw[ultra thick] (13)--(14);
    \draw[ultra thick] (13)--(15);
    \draw[ultra thick] (14)--(15);
    
    \node[] at (2,3) {$A$};
    \end{scope}
\end{tikzpicture}
\caption{A disk $D$ and its boundary complex $A$}
\label{figone}
\end{figure}
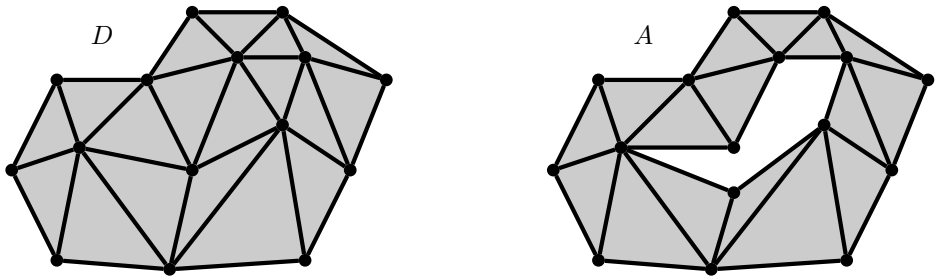

\begin{lem}\label{annulus}
Let $\varphi:D \to X$ be a minimal filling diagram for a cycle $\sigma$ in a locally large length complex $X$, where $D$ has at least two interior vertices, and no edge connecting nonconsecutive vertices of the boundary. Let $A$ be the boundary complex of $D$. Then:
$$l(\partial_1A) \geq l(\partial_2A)+2.$$
\end{lem}
\begin{proof}
We apply Gauss-Bonnet to $D$ and $A$ to obtain (after simplifying the notation of the indices of the sums)
$$2 = \sum_{f\in D} \kappa(f) + \sum_{v\in D} \kappa(v) \leq \sum_{f\in A} \kappa(f) + \sum_{v\in \partial D} \kappa(v) = \sum_{f\in A} \kappa(f) + \sum_{v\in \partial_1A} \kappa(v),$$
$$0 = \sum_{f\in A} \kappa(f) + \sum_{v\in \partial_1A\cup\partial_2A} \kappa(v).$$
The first and last equalities hold because the Euler characteristic of a disk and an annulus are $1$ and $0$ respectively. The first inequality is due to the fact that the curvature of faces and interior vertices is always nonpositive.

By taking the double of the first expression and subtracting the second expression we get
$$4 \leq 2\left(\sum_{f\in A} \kappa(f) + \sum_{v\in \partial_1A}\kappa(v)\right) - \sum_{f\in A} \kappa(f) - \sum_{v\in \partial_1A\cup\partial_2A}\kappa(v) = \sum_{f\in A} \kappa(f) + \sum_{v\in \partial_1A} \kappa(v) - \sum_{v\in\partial_2A} \kappa(v).$$

Observe that the Euler characteristic of the link of a vertex in the boundary is equal to $1$. We note by $F_1$ and $F_2$ the sets of $2$-simplices of $A$ having one edge in $\partial_1A$ and $\partial_2A$ respectively. For a face $f$ in $F_1$ or $F_2$ we denote its three sides by $e_1^f$, $e_2^f$ and $e_3^f$, where $e_1^f$ is the one lying in the corresponding boundary component. We also denote their respective lengths by $l_1^f$, $l_2^f$ and $l_3^f$. Note that the cardinality of $F_i$ is $|\partial_iA|$ for $i=1,2$. From this we have
\begin{equation*} 
\begin{split}
4 & \leq \sum_{f\in A} \left(\left(\sum_{e\in\partial f}l(e)\right) - 1\right) + \sum_{v\in \partial_1A} \left(1 - \sum_{e\in Lk_A(v)} l(e)\right) - \sum_{v\in\partial_2A} \left(1 - \sum_{e\in Lk_A(v)} l(e)\right) \\
& = \sum_{f\in F_1}\left(\left(\sum_{e\in \partial f} l(e)\right) -1 +1+l_1^f-l_2^f-l_3^f \right) + \sum_{f\in F_2}\left(\left(\sum_{e\in \partial f} l(e)\right) -1 -1 -l_1^f+l_2^f+l_3^f \right)\\
& = \sum_{f\in F_1}2l_1^f + \sum_{f\in F_2}\left(-2+2l_2^f+2l_3^f \right)\\
& \leq \sum_{f\in F_1}2l_1^f + \sum_{f\in F_2}-2l_1^f \\
& = 2l(\partial_1A) - 2l(\partial_2A).
\end{split}
\end{equation*}
The first equality is a rearrangement of the terms using the new notation and the remark about the cardinalities of the $F_i$. The last inequality holds because the sum of the lengths of the sides of any $2$-simplex is less than or equal to 1. Dividing both sides by 2, we obtain the desired inequality $l(\partial_1A) \geq l(\partial_2A)+2$.
\end{proof}

We now have all the necessary ingredients to prove Theorem \ref{localtoglobal}.

\begin{proof}[Proof of Theorem \ref{localtoglobal}]
We have to show that every full cycle has length greater than or equal to~ 2, and that $X$ is flag. Let $\sigma$ be a full cycle in $X$. Since $X$ is simply connected, by Lemma \ref{existdiagram} there is a minimal filling diagram for $\sigma$, say $\varphi:D\to X$. We know that $\varphi$ is nondegenerate because it is minimal. Hence by Lemma \ref{nicediagram}, $D$ is locally large.
Since $\sigma$ is full, there are no edges in $D$ connecting nonconsecutive vertices of its boundary, and $D$ has at least one interior vertex. If $D$ has only one interior vertex $v$ we have
$$0\geq \kappa (v) = 2-\sum_{e\in Lk_D(v)} l(e) = 2-l(\sigma).$$
Therefore $l(\sigma)\geq 2$. If $D$ has more than one interior vertex, then we are under the hypotheses of Lemma \ref{annulus}, and $l(\sigma) = l(\partial_1A) \geq 2$.

Now we show that $X$ is flag. We are going to see that it suffices to show that every cycle with three edges spans a $2$-simplex in $X$. Indeed, suppose we have vertices $v_1,\dots,v_n$ that are pairwise connected. If every triangle is filled, we have the $1$-skeleton of an $(n-1)$-simplex in $Lk_X(v_1)$. Since the links of the vertices are flag, $v_1,\dots,v_n$ must span an $n$-simplex in $X$.

Take a cycle $\sigma$ with three edges, and let $\varphi:D\to X$ be a minimal filling diagram for $\sigma$. If $D$ has more than one interior vertex, then by Lemma \ref{annulus}, $l(\sigma)\geq 2$, which is impossible. If $D$ has exactly one interior vertex $v$, then just as before
$$0\geq \kappa (v) = 2-\sum_{e\in Lk_D(v)} l(e) = 2-l(\sigma).$$
Once again, this would imply that $l(\sigma)\geq 2$. So the only possibility is that $D$ has no interior vertices. Hence, $\sigma$ spans a $2$-simplex in $X$.
\end{proof}


\section{The Artin complex}\label{theartincomplex}

In this section we recall the construction of the complex associated to an Artin group described in \cite{CMV}. We follow their description and maintain their notation. The definitions and notations related to complexes of groups are those of \cite[Chapter II.12]{BH}.

Let $A_S$ be an Artin group with generator set $S$ (with $|S|\geq 2$). Take $K$ a simplex of dimension $|S|-1$ and define a simplex of groups over $K$. First, give the simplex $K$ a trivial local group. Simplices of codimension $1$ are in one-to-one correspondence with elements $s_i\in S$, and are denoted by $\Delta_{s_i}$. The simplex $\Delta_{s_i}$ is given the local group $\langle s_i \rangle$. Now every simplex of codimension $k$ is in one-to-one correspondence with a subset of $S$ of cardinality $k$. Given $S'\subset S$ with $|S'|=k$, its corresponding face can be written uniquely as
$$\Delta_{S'} = \cap_{s_i\in S'}\Delta_{s_1}.$$
The simplex $\Delta_{S'}$ is given the local group $A_{S'}$.

Given an inclusion $\Delta_{S''}\subset \Delta_{S'}$ there is a natural inclusion $\psi_{S'S''}:A_{S'}\xrightarrow{} A_{S''}$. Let $\mathcal{P}$ be the poset of standard parabolic subgroups of $A_S$ with the order given by the natural inclusions. Since every standard parabolic subgroup is itself an Artin group \cite{V}, there is a simple morphism $\varphi: G(\mathcal{P}) \xrightarrow{} A_S$, given by inclusion, from the complex of groups to $A_S$. 

\begin{definition}\label{artincomplex} The \textit{Artin complex} associated to $A_S$ is the development $X_S:= D_K(\mathcal{P},\varphi)$ of $\mathcal{P}$ over $K$ along $\varphi$ (\cite[Theorem II.12.18]{BH}).
\end{definition}

In the proof of \cite[Theorem II.12.18]{BH}, an explicit description of $X_S$ is given. The simplicial complex $X_S$ can be defined as
$$X_S:=A_S\times K /\sim,$$
where $(g,x)\sim (g',x')$ if and only if $x=x'$ and $g^{-1}g'$ is in the local group of the smallest simplex of $K$ containing $x$.

The action of $A_S$ in $X_S$ is by simplicial isomorphisms, without inversions and cocompact, with strict fundamental domain $K$. Any simplex $\Delta$ of $X_S$ is in the orbit of exactly one $\Delta_{S'}\subset K$ for some $S'\subset S$. In that case $\Delta$ is said to be \textit{of type} $S'$. We now recall some results from \cite{CMV} about the complex $X_S$.

\begin{lem}[\cite{CMV}, Lemma 4]\label{simplyconnected} Let $A_S$ be an Artin group and let $X_S$ be its Artin complex. Then $X_S$ in connected. Additionally, if $|S|\geq 3$, then $X_S$ is simply connected.
\end{lem}

\begin{lem}[\cite{CMV}, Lemma 6]\label{links} Let $A_S$ be an Artin group with Artin complex $X_S$. The link of a simplex of type $S'$ is isomorphic to the Artin complex $X_{S'}$ associated to the Artin group $A_{S'}$. 
\end{lem}

\begin{lem}[\cite{CMV}, Lemma 9]\label{systole} Let $A_S$ be an Artin group with $S=\{s,t\}$. Then any cycle in $X_S$ has at least $2m_{st}$ edges, and it it a tree if $m_{st}=\infty$.
\end{lem}

\begin{remark}
In \cite{CMV}, they show the previous result for $m_{st}\in\{3,4,\dots,\infty\}$, since they work with large-type Artin groups. However, the result also holds for the case $m_{st}=2$, and the proof is the same as in the other cases.
\end{remark}


\section{Parabolic subgroups}\label{parabolicsubgroups}

We now  define a length function for the Artin complex of a given $(2,2)$-free two-dimensional Artin group, and show that it is a systolic-by-function length complex. Then we make use of its geometric structure to prove Theorem \ref{theorem}.

\begin{definition}
An Artin group $A_S$ is \textit{$(2,2)$-free} if its defining graph does not have two consecutive edges labeled by 2. 
\end{definition}

Notice that large-type Artin groups are $(2,2)$-free two-dimensional Artin groups. Now we give an equivalent characterization.

\begin{prop}\label{labels}
A two-dimensional Artin group $A_S$ is $(2,2)$-free if and only if there exist numbers $m'_{st}\in \{2,3,4,6\}$ with $m'_{st}\leq m_{st}$ and $m'_{st}=m'_{ts}$ for every $s,t\in S$, such that $\frac{1}{m'_{st}}+\frac{1}{m'_{tr}}+\frac{1}{m'_{rs}}\leq 1$ and $\frac{1}{m'_{st}}\leq\frac{1}{m'_{tr}}+\frac{1}{m'_{rs}}$ for every $r,s,t\in S$.
\end{prop}

\begin{proof}
It is clear that if such $m'_{st}$ exist, then $A_S$ is $(2,2)$-free. Now suppose that the graph of $A_S$ does not have two consecutive edges labeled by 2. The we can define the $m'_{st}$ in the following way:
\begin{itemize}
\item if $m_{st}=2$, then $m'_{st}=2$;
\item if $m_{st}=3$, then $m'_{st}=3$;
\item if $m_{st}>3$ and the edge is not adjacent to an edge labeled by $2$, then $m'_{st}=3$;
\item if $m_{st}$ forms a triangle with a 2 and a 3, then $m'_{st}=6$;
\item in any other case, $m'_{st}=4$.
\end{itemize} 

It is easy to see that such labeling is well defined since $A_S$ is a $(2,2)$-free two-dimensional Artin group, and that it satisfies the required conditions.
\end{proof}

Let $A_S$ be a $(2,2)$-free two-dimensional Artin group. We define a length function for $X_S$, $l:\text{edges}(X_S) \to [0,\frac{1}{2}]$ as follows. Edges in $X_S$ are simplices of codimension $|S|-2$, so they correspond to subsets of $S$ that are missing two elements. We define the length of an edge of type~ $S\backslash\{s,t\}$ to be $\frac {1}{m'_{st}}$, where the $m'_{st}$ are the ones in the previous proposition. Since $A_S$ is $(2,2)$-free and two-dimensional, the sum of the lengths of the three edges of every triangle is less than or equal to 1, and the triangle inequality holds, so $l$ is well defined.

\begin{thm}\label{artinlsystolic}
Let $A_S$ be a $(2,2)$-free two-dimensional Artin group with $|S|\geq 3$. Then $X_S$ with the length function defined as above is systolic-by-function.
\end{thm}

\begin{proof}
The proof proceeds by induction on $|S|$ by using our local-to-global Theorem \ref{localtoglobal}.

If $|S|=3$, by Lemma \ref{simplyconnected} $X_S$ is connected and simply connected. Let $v$ be a vertex of $X_S$. Lemma~ \ref{links} says that $Lk_{X_S}(v)$ is isomorphic to the Artin complex $X_{S'}$ associated to the Artin group $A_{S'}$, where $S'\subset S$ with $|S'|=2$. The complex $X_{S'}$ is a graph and it inherits the length function from $X_S$. If $X_{S'}$ is a tree, then it is clearly a large length complex. If it is not a tree, then by Proposition \ref{labels}, the length of its edges is greater than or equal to $\frac{1}{m}$, where $m$ is the label in A$_{S'}$. By Lemma~ \ref{systole} all cycles in $X_{S'}$ have at least $2m$ edges, so it is a large length complex. Thus $X_S$ is systolic-by-function.

Now assume that $|S|>3$ and that the claim holds for every $(2,2)$-free two-dimensional Artin group with less generators. Once again, we know $X_S$ is connected and simply connected from Lemma \ref{simplyconnected}. Applying Lemma \ref{links} we get that the link of every vertex is systolic-by-function. Hence, by Theorem \ref{localtoglobal}, the link of every vertex is large. Therefore $X_S$ is systolic-by-function.
\end{proof}

\begin{remark}
With the length function defined as above, one can give $X_S$ a metric such that it is metrically systolic in the sense of \cite{HO}. Concretely, if an edge $e\in X_S$ has $l(e)=\frac{1}{k}$, then the length of $e$ in the metric is $\sin(\frac{\pi}{k})$. We will not use metric systolicity in this article, but this fact may be of interest for other applications.
\end{remark}

We define the \textit{distance} between two vertices $u,v \in X_S$ as
$$d(u,v) = \min\{l(\gamma) \mid \gamma \text{ is a path connecting } u \text{ and } v\}.$$
Note that this minimum is attained, because $X_S$ is connected and the image of $l$ is a finite subset of $[0, \frac{1}{2}]$. We say that a path $\gamma$ between $u$ and $v$ is a \textit{geodesic} if it is of minimum length.

We are now ready to prove Theorem \ref{theorem}.

\begin{proof}[Proof of Theorem \ref{theorem}]
We want to show that if an element $g\in A_S$ fixes two vertices in $X_S$, then there is a path between said vertices that is fixed pointwise by $g$. Suppose it is not the case. Take vertices $u$ and $v$, and $g\in A_S$ such that the condition fails and such that $d(u,v)$ is minimal among such pairs. Let $\gamma$ be a geodesic between them (it exists because $X_S$ is connected). Then $g$ maps $\gamma$ to another geodesic $\gamma'$ between $u$ and $v$. Since $d(u,v)$ is minimal, the union of $\gamma$ and $\gamma'$ determines a cycle in $X_S$. We will show that we can fill the cycle with a minimal filling diagram and find a shortcut between $u$ and $v$, contradicting the fact that $\gamma$ is a geodesic.

Let $\varphi: D\to X_S$ be a minimal filling diagram for the concatenation of $\gamma$ and $\gamma'$ (it exists because $X_S$ is simply connected). We label the vertices of $\gamma$ in $D$ as $u=v_0,v_1,\dots,v_{|\gamma|-1}, v_{|\gamma|} = v$, and the vertices of $\gamma'$ in $D$ as $u=v'_0,v'_1,\dots,v'_{|\gamma'|-1}, v'_{|\gamma'|} = v$.  Note that
\begin{enumerate}
\item we may assume, without loss of generality, that there is no edge between nonconsecutive vertices of $\gamma$, or between nonconsecutive vertices of $\gamma'$, because $\gamma$ is a geodesic, and
\item there is no edge between $v_i$ and $v'_i$ for $1\leq i \leq |\gamma|-1$, because vertices of the same type are not connected by an edge in $X_S$.
\end{enumerate}

However, there may be an edge between $v_i$ and $v'_j$ if $i \neq j$. Take the rightmost of these edges. We call it $e$ and assume that it connects $v_k$ with $v'_{k+r}$ (see Figure \ref{figthree}). Let $\tilde{D}$ be the disk delimited by $e$, and let $\tilde{\gamma}$ and $\tilde{\gamma'}$ be the paths connecting $v_k$ and $v$ along the boundary of $\tilde{D}$, where $\tilde{\gamma'}$ contains $e$. If there is no such edge $e$, we consider $\tilde{D}=D$, $\tilde{\gamma}=\gamma$ and $\tilde{\gamma'}=\gamma'$, and continue in the same way. It is clear that $\tilde{\gamma}$ is a geodesic. Then $l(e)\geq d(v_k,v_{k+r})$. We also have that $l(e)\leq d(v_k,v_{k+r})$. To see so, we project both $e$ and the path in $\tilde{\gamma}$ between $v_k$ and $v_{k+r}$ to the fundamental domain $K$, and apply the triangle inequality $r-1$ times. Hence, $l(\tilde{\gamma'}) = l(\tilde{\gamma})$. From the previous observations, $\tilde{D}$ has no edges connecting nonconsecutive boundary vertices. Therefore, it has at least one interior vertex. As in the proof of Theorem \ref{localtoglobal} we get that $l(\partial \tilde{D}) = l(\tilde{\gamma}) + l(\tilde{\gamma'}) \geq 2$. 

\begin{figure}[h]
\begin{tikzpicture}[scale=0.6]

\tikzstyle{point}=[circle,thick,draw=black,fill=black,inner sep=0pt,minimum width=4pt,minimum height=4pt]
    \node (0)[point] at (0,0) {};
    \node (1)[point] at (1,1) {};
    \node (1')[point] at (1,-1) {};
    \node (2)[point] at (2,1) {};
    \node (2')[point] at (2,-1) {};
    \node (3)[point] at (6,1) {};
    \node (3')[point] at (6,-1) {};
    \node (4)[point] at (7,1) {};
    \node (4')[point] at (7,-1) {};
    \node (5)[point] at (8,1) {};
    \node (5')[point] at (8,-1) {};
    \node (6)[point] at (9,1) {};
    \node (6')[point] at (9,-1) {};
    \node (7)[point] at (10,1) {};
    \node (7')[point] at (10,-1) {};
    \node (8)[point] at (11,0) {};
    
    \draw[ultra thick] (0)--(1);
    \draw[ultra thick] (0)--(1');
    \draw[ultra thick] (1)--(2);
    \draw[ultra thick] (1')--(2');
    \draw[ultra thick] (3)--(4);
    \draw[ultra thick] (3')--(4');
    \draw[ultra thick] (4)--(5);
    \draw[ultra thick] (4')--(5');
    \draw[ultra thick] (5)--(6);
    \draw[ultra thick] (5')--(6');
    \draw[ultra thick] (6)--(7);
    \draw[ultra thick] (6')--(7');
    \draw[ultra thick] (7)--(8);
    \draw[ultra thick] (7')--(8);
    \draw[ultra thick] (1')--(2);
    \draw[ultra thick] (3')--(4);
    \draw[ultra thick] (4)--(6');
    \draw[ultra thick, loosely dotted] (2)--(3);
    \draw[ultra thick, loosely dotted] (2')--(3');
    
    \node[] at (7.5,0) {$e$};
    \node[above] at (4) {$v_k$};
    \node[below] at (6') {$v'_{k+r}$};
    \node[left] at (0) {$u$};
    \node[right] at (8) {$v$};
    \node[] at (1.5,1.5) {$\gamma$};
    \node[] at (1.5,-1.5) {$\gamma'$};
    \node[] at (4,0) {$D$};
    \node[above right] at (6) {$\tilde{\gamma}$};
    \node[below right] at (7') {$\tilde{\gamma'}$};
    
    \draw [fill=black, fill opacity=0.2] (7,1)--(9,1)--(10,1)--(11,0)--(10,-1)--(9,-1)--cycle;
    
    \node[] at (9.5,0) {$\tilde{D}$};
\end{tikzpicture}
\caption{Disk $D$, the rightmost edge $e$ and disk $\tilde{D}$}
\label{figthree}
\end{figure}
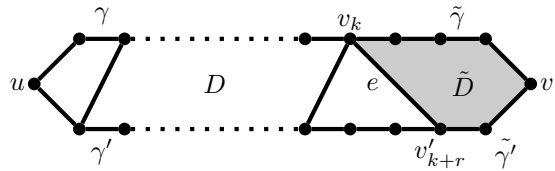

If either $v_k$ or $v$ have degree greater than 3, then by Lemma \ref{annulus} we could find a path in $\tilde{D}$ connecting $v_k$ and $v$ shorter than $\tilde{\gamma}$. That would contradict the fact that $\gamma$ is a geodesic. So we can assume that both $v_k$ and $v$ have degree 3 (it is clear that they cannot have degree 2). Let $e_1$ and $e_2$ be the interior edges incident to $v_k$ and $v$ respectively. If either of them has length less than $\frac{1}{2}$, then either by Lemma \ref{annulus}, in case $\tilde{D}$ has more than one interior vertex, or by the inequality $l(\partial \tilde{D})\geq2$ if $\tilde{D}$ has exactly one interior vertex, we can find a shortcut and get a contradiction. 

The only situation remaining is when $v_k$ and $v$ have degree 3, and $l(e_1)=l(e_2)=\frac{1}{2}$. In that case we can find a path $\sigma$ in $\tilde{D}$ starting at $v_k$ and ending with $e_2$, with $l(\sigma)\leq l(\tilde{\gamma})$. If the inequality is strict the proof is finished. If $l(\sigma) = l(\tilde{\gamma})$, consider the geodesic $\tau$ from $u$ to $v$ consisting of concatenating the subpath of $\gamma$ that goes from $u$ to $v_k$ with $\varphi(\sigma)$. We can assume $\varphi(\sigma)$ is a path in $X_S$, because otherwise we would have found a shortcut. Since this geodesic ends with an edge of length $\frac{1}{2}$ and no triangle has two edges of length $\frac{1}{2}$, applying the same procedure to $\tau$ gives us a path between $u$ and $v$ shorter than $\gamma$, obtaining the desired contradiction.
\end{proof}

\begin{proof}[Proof of Theorem \ref{main}]
It follows immediately from Theorems \ref{parabolic} and \ref{theorem}.
\end{proof}

Applying this theorem together with previous results we get following corollary.

\begin{cor}
Let $A_S$ be an Artin group with at most three generators. Then the intersection of an arbitrary family of parabolic subgroups is a parabolic subgroup.
\end{cor}
\begin{proof}
Such an Artin group is either spherical type, right-angled or $(2,2)$-free two-dimensional. Therefore, either by \cite{CGGW} for the spherical case; by \cite{AM} for the right-angled case; or by Theorem \ref{main} for the $(2,2)$-free two-dimensional case, we get the desired result.
\end{proof}

We finish this article by applying our results and the algorithm introduced in \cite[Algorithm 4]{C} to solve the conjugacy stability problem for $(2,2)$-free two-dimensional Artin groups. We follow the notation and definitions of \cite{C}.

\begin{thm}[\cite{C}, Theorem A]\label{cumplido}
Let $A_S$ be a standardisable Artin group satisfying the ribbon property and such that every element in $A_S$ has a parabolic closure. Then, there is an algorithm that decides if a parabolic subgroup $P$ of $A_S$ is conjugacy stable or not.
\end{thm}

It is known by results of Godelle \cite{G} that two-dimensional Artin groups are standardisable and satisfy the ribbon property. Also, by Theorem \ref{theorem} and Theorem \ref{parabolic}, any element in a $(2,2)$-free two-dimensional Artin group has a parabolic closure. Therefore, $(2,2)$-free two-dimensional Artin groups satisfy the hypothesis of Theorem \ref{cumplido}. By examining the aforementioned algorithm in the $(2,2)$-free two-dimensional case we obtain the classification of Theorem \ref{conj}.

\begin{proof}[Proof of Theorem \ref{conj}]
We need to understand \cite[Algorithm 4]{C} for a $(2,2)$-free two-dimensional Artin group $A_S$ and a standard parabolic subgroup $A_{S'}$. Since $A_S$ is two-dimensional, the only spherical-type parabolic subgroups are dihedral Artin groups of type $I_2(m)$. Thus, the algorithm reduces to checking if there exist vertices $s,t$ in $\Gamma_{S'}$ that are connected by an odd-labeled path in~ $\Gamma_S$, but are not connected by an odd-labeled path in $\Gamma_{S'}$. This is exactly the criterion we wanted to prove.
\end{proof}

For a more detailed proof of this fact, see \cite[Theorem C]{CMV}. Their proof is for large-type Artin groups, but it also works in the $(2,2)$-free two-dimensional case.


\end{document}